\documentclass[10pt,reno]{amsart} 
\usepackage{amssymb}
\usepackage{amsmath}
\usepackage{enumitem}
\usepackage{mathrsfs}
\usepackage[english]{babel}
\usepackage[all]{xy}
\usepackage{hyperref}
\usepackage{marginnote}

\usepackage{stmaryrd}

\usepackage[margin=1.10in]{geometry}

\RequirePackage{mathrsfs} 

\newtheorem{theorem}{Theorem}[section]
\newtheorem{lemma}[theorem]{Lemma}
\newtheorem{proposition}[theorem]{Proposition}

\newtheorem{conjecture}[theorem]{Conjecture}

\theoremstyle{definition}
\newtheorem*{ack}{Acknowledgements}
\newtheorem*{con}{Conventions}
\newtheorem{remark}[theorem]{Remark}

\newtheorem{definition}[theorem]{Definition}

\numberwithin{equation}{section} \numberwithin{figure}{section}

\DeclareMathOperator{\Spec}{Spec}

\DeclareMathOperator{\an}{an}

\DeclareMathOperator{\Hom}{Hom}

\newcommand{\thickslash}{\mathbin{\!\!\pmb{\fatslash}}}

\newcommand\ZZ{\mathbb{Z}}

\newcommand\QQ{\mathbb{Q}}

\newcommand\CC{\mathbb{C}}

\usepackage{color}

\definecolor{orange}{rgb}{1,0.5,0}

\title[Finiteness of pointed families of   polarized varieties]{The Shafarevich conjecture revisited: Finiteness of pointed families of   polarized varieties}

\author{Ariyan Javanpeykar}
\address{Ariyan Javanpeykar \\
Institut f\"{u}r Mathematik\\
Johannes Gutenberg-Universit\"{a}t Mainz\\
Staudingerweg 9, 55099 Mainz\\
Germany.}
\email{peykar@uni-mainz.de}

 \author{Ruiran Sun}
 \address{Ruiran Sun \\
Institut f\"{u}r Mathematik\\
Johannes Gutenberg-Universit\"{a}t Mainz\\
Staudingerweg 9, 55099 Mainz\\
Germany.}
\email{ruirasun@uni-mainz.de}

\author{Kang Zuo}
 \address{Kang Zuo \\
Institut f\"{u}r Mathematik\\
Johannes Gutenberg-Universit\"{a}t Mainz\\
Staudingerweg 9, 55099 Mainz\\
Germany.}
\email{zuok@uni-mainz.de}

\subjclass[2010]
{32Q45, 
14G99 
(11G35,  
14G05,  
14C30) 
}

\keywords{Hyperbolicity, Lang-Vojta conjecture, moduli spaces, Higgs bundles, deformation theory, families of polarized varieties, Arakelov inequality, integral points,  number fields,  finitely generated fields}

\begin{document}

\begin{abstract}  Motivated by Lang-Vojta's conjectures on hyperbolic varieties, we prove a new version of the  Shafarevich conjecture in which we establish the finiteness of    pointed families of     polarized varieties. We then give an arithmetic application  to the finiteness of integral points on   moduli spaces of   polarized varieties.  
\end{abstract}

\maketitle

\thispagestyle{empty}
 
 \section{Introduction}

 Let $k$ be an algebraically closed field of characteristic zero.
   In \cite{ArakelovShaf, Parshin3} Arakelov-Parshin proved  Shafarevich's 1962 ICM conjecture for the moduli space of curves \cite{Shaf1962}. Namely, for every smooth connected curve $B$ over $k$ and every integer $g\geq 2$, the set of $B$-isomorphism classes of non-isotrivial smooth proper curves $X\to B$  of genus $g$ is finite, i.e., the set of isomorphism classes of non-isotrivial morphisms $B\to \mathcal{C}_g\otimes_{\QQ} k$ is finite, where $\mathcal{C}_g$ is the stack over $\mathbb{Q}$ of smooth proper   connected curves of genus $g$, and $\mathcal{C}_g\otimes_{\QQ} k $ denotes the stack $\mathcal{C}_g\times_{\Spec \QQ} \Spec k$.
 
In this paper we      prove a generalization of Arakelov-Parshin's finiteness result in the case of higher-dimensional families of    polarized varieties; see Theorem \ref{thm:main} for a precise statement.   

Our motivation is twofold. First, our finiteness result fits in well with predictions made by the Lang-Vojta conjectures on finiteness properties of hyperbolic varieties. Second, we use our finiteness result to verify the Persistence Conjecture (see Theorem \ref{thm2}) for the moduli space of canonically polarized varieties. This latter result is a crucial ingredient in the proof of   a new arithmetic  finiteness result for smooth hypersurfaces in abelian schemes over arithmetic rings proven  in \cite{JM} (which generalizes the work of Lawrence-Sawin \cite{LawrenceSawin} from number fields to finitely generated fields).

\subsection{Main result}   
To state our result, let $\mathrm{Pol}$ be the stack  over $\QQ$ of   polarized smooth proper  connected varieties with semi-ample canonical bundle.  (This is a substack of the larger stack considered in \cite[Tag~0D1L]{stacks-project}.) The objects of   $\mathrm{Pol} $    are tuples    $(S,f\colon X\to S, \mathcal{L})$ with $S$ a scheme over $\mathbb{Q}$,  $f\colon X\to S$  a smooth proper morphism of schemes whose geometric fibres have semi-ample canonical bundle, and $\mathcal{L}$ an $f$-relatively ample line bundle.   A morphism from  $(S, f\colon X\to S, \mathcal{L})$ to $(S',f'\colon X'\to S', \mathcal{L}')$  in $\mathcal{M}(S)$ is a triple $(\phi_1,\phi_2, \psi)$ where $\phi_1\colon X\to X'$ is a morphism, $\phi_2\colon S\to S'$ is a morphism and $\psi\colon \phi_1^\ast \mathcal{L}'\to \mathcal{L}$ is an isomorphism such that the diagram
\[
\xymatrix{ X \ar[rr]^{\phi_1} \ar[d]_{f} & & X' \ar[d]_{f'} \\ S\ar[rr]_{\phi_2} & & S'.}
\]
is Cartesian.  By \cite[Tag~0D4X]{stacks-project}, the stack $\mathcal{M}$ is algebraic. Its diagonal is affine and of finite type \cite[Lemma~3.1]{JLFano}.   

We let $\mathcal{M} := \mathrm{Pol}\thickslash G$ be the rigidification of $\mathrm{Pol}$ by the flat group substack $G= \mathbb{G}_{m,\mathrm{Pol}}$ of the inertia stack $I_{\mathcal{M}}$ (corresponding to automorphisms given by only scaling the line bundle); see \cite[Theorem~A.1]{AOV} for details on rigidification.  
 For $h$ in $\QQ[t]$, let $\mathcal{M}_h$ be the substack parametrizing   polarized varieties  with Hilbert polynomial $h$. Then, by the theory of Hilbert schemes, Matsusaka's Big Theorem \cite{Mat72}, and the theorem of Matsusaka-Mumford \cite{MatMum}, the stack $\mathcal{M}_h$ is a    finite type separated Deligne-Mumford   algebraic stack     over $\mathbb{Q}$ and $\mathcal{M}=\sqcup_{h\in \mathbb{Q}[t]} \mathcal{M}_h$.     

  Note that, for $g\geq 2$, the stack $\mathcal{C}_g$ of smooth proper connected curves of genus $g$ is a connected component of  $\mathcal{M}$.    Also, the stack $\mathcal{M}$ contains the stack of polarized Calabi-Yau varieties, polarized abelian varieties, and also the stack of canonically polarized varieties (i.e., those with ample canonical bundle).

Given a   variety $U$ over $k$,  the data of a morphism $U\to \mathcal{M}\otimes_{\QQ} k$ is equivalent to the data of   a smooth proper morphism  $\pi:V\to U$ whose geometric fibres are connected with semi-ample canonical bundle, and the data $L$ of a relatively ample line bundle on $U$.  (We denote this as $(V,L)\to U$.)    Thus,  by definition,
  if $U$ is a variety over $k$ and  $U\to \mathcal{M}\otimes_{\QQ} k$ is a morphism and $(V, L)\to U$ is the associated family, then $U\to \mathcal{M}\otimes_{\QQ} k$ is quasi-finite (as a morphism of algebraic stacks) if and only if, for every   polarized variety $(V_0,L_0)$, the set of $u$ in $U(k)$ such that 
\[
(V_u,L_u)\cong (V_0,L_0)
\] is finite.

We prove the following version of the Shafarevich conjecture which establishes the finiteness of families over a given variety with a    \emph{large enough} fixed set of fibres.  To state our result,  for $C$  a smooth quasi-projective connected curve, we let $\overline{C}$ be its smooth projective model and $g(\overline{C})$ its genus.  

\begin{theorem}[Weak-Pointed Shafarevich Conjecture]\label{thm:main} 
Let $k$ be an algebraically closed field of characteristic zero, let $h\in \QQ[t]$, let $U$ be a quasi-projective variety over $k$, and let $U\to \mathcal{M}_{h}\otimes_{\QQ} k$ be a quasi-finite morphism.   Then $U$ has the following finiteness property:

If  $C$  is a smooth quasi-projective connected curve,  $N\geq 0$ is an integer,   $c_1,\ldots, c_N\in C(k)$  are pairwise distinct points, $u_1,\ldots, u_N\in U(k)$ are points, and 
 $$N\geq \frac{\deg h-1}{2}\left(2g(\overline{C}) -2 +\# (\overline{C}\setminus C)\right),$$ 
 then   the set of non-constant morphisms $f:C\to U$ with $f(c_1) = u_1, \ldots, f(c_N) = u_N$ is finite.
\end{theorem}
 
  The study of families of varieties with enough ``fixed'' fibres is not artificial and appears naturally in  the study of finiteness properties of hyperbolic varieties; see Section \ref{section:lang} below for a discussion. 
  
  The conclusion of our theorem (Theorem \ref{thm:main}) probably holds   for all \textbf{hyperbolic} varieties; see Conjecture \ref{conj:lv}. However, proving this general expectation is currently out-of-reach (much like its arithmetic analogue predicting finiteness of integral points on hyperbolic varieties).

  In   Theorem \ref{thm:main} one may take ``$N=0$'' when $\deg h = 1$. In particular, for $g\geq 2$, by  applying Theorem \ref{thm:main} to   the fine moduli space  $U:=\mathcal{C}_g^{[3]}$ of  smooth proper  curves of genus $g$ with full level $3$ structure we recover, up to a standard stacky argument, the finiteness theorem of   Arakelov-Parshin for $\mathcal{C}_g$; see Remark \ref{remark:ap} for details.

The Shafarevich conjecture for families of curves was proved by Arakelov-Parshin in two steps. First, they showed that the moduli space of   smooth proper curves of genus $g$ over a fixed base curve is bounded   (i.e., it is of finite type) using Arakelov's inequality. Then, they showed that the moduli space of non-isotrivial  curves over $B$ is zero-dimensional (i.e., each non-isotrivial family of higher genus curves is   ``rigid'').   However, the analogue of the Shafarevich conjecture for higher-dimensional families  fails, as non-isotrivial families of such varieties can  be non-rigid. Let us be more precise.
  
  One can not take ``$N=0$'' in Theorem \ref{thm:main} when $\deg h >1$, as there are many curves $C$ which admit a non-constant \emph{non-rigid} morphism $C\to \mathcal{M}\otimes_{\QQ} \mathbb{C}$. For example,  let $g\geq 2$ be an integer and let $C\to \mathcal{C}_g$ be a non-isotrivial morphism (associated to a non-isotrivial smooth proper curve of genus $g$ over $C$). Define $U=C\times C$. Then $U$ admits a quasi-finite morphism to $\mathcal{C}_g\times \mathcal{C}_g$, and therefore also to $\mathcal{M}$. However, the set of non-constant morphisms $C\to U=C\times C$ is clearly not finite.   
  
Motivated by Lang-Vojta type conjectures, we expect that families of higher-dimensional   polarized varieties with semi-ample canonical bundle are rigid, as long as one fixes  ``enough'' fibres of the family.  We formulate our most optimistic expectation below (see Conjecture \ref{conj}).

To prove  our main finiteness result (Theorem \ref{thm:main})  we prove a new rigidity theorem for families of polarized varieties; see   Theorem \ref{thm:pointed_rigidity}. In our proof we use  Viehweg-Zuo sheaves and   Higgs bundles. Although these tools are   standard by now, their use in the study of pointed maps is novel.
%
%
%
   
    Note that in Theorem \ref{thm:main} we do not assume $U$ to be smooth, nor even normal. Indeed, imposing smoothness or normality of $U$ would be unnatural, as the singularities of the moduli stack of canonically polarized varieties  satisfy Murphy's Law; see \cite{Vakil}. The fact that $U$ is allowed to be singular complicates the proof at certain points; we refer the reader to Remarks \ref{remark:xi} and \ref{remark:singular} for a brief discussion of these points.
  
Finally, we note that, in the presence of an infinitesimal Torelli theorem for the varieties parametrized by $\mathcal{M}_h$ the conclusion of Theorem \ref{thm:main} can be deduced from finiteness properties of varieties with a quasi-finite period map; see    \cite[Theorem~1.5]{JLitt} for details. However, it is well-known that there are canonically polarized surfaces for which   infinitesimal Torelli   fails.

We have also proven a related, but different finiteness result for pointed maps (using entirely different methods) in joint work with Steven Lu\cite{JLSZ}.  For the reader's convenience, we summarize the differences here.

  \begin{enumerate}
  \item In \cite{JLSZ}, we prove   finiteness of maps $(Y,y)\to (U,u)$ for every pointed variety $(Y,y)$ and every point $u$  in $U$ \textbf{not} in some proper closed subset $\Delta$ (which is independent of $Y$ and $y$). In this paper we prove finiteness of maps $(C,c_1,\ldots,c_N) \to (U,u_1,\ldots, u_N)$ for all choices of pairwise distinct points $c_1,\ldots, c_N$ and all choices of points $u_1,\ldots, u_N$ in $U$, assuming $N$ is "large enough".
  \item In \cite{JLSZ}, our rigidity result for pointed maps is a consequence of a more general rigidity property for varieties with a certain Finsler pseudometric; the existence of the   Finsler pseudometric on the moduli space of polarized varieties necessary in \cite{JLSZ}  was proven by    Ya Deng (in the generality applied there). On the other hand, in  this paper we rely on an Arakelov-type inequality for the number of base points. Both proofs have in common that they rely on constructions of Viehweg-Zuo. However, in  \cite{JLSZ} we do not require any \emph{new} Viehweg-Zuo type constructions. In the current paper, we have to adapt existing Viehweg-Zuo-type results and constructions to the setting of $N$-pointed maps (which is why a large part of this paper  is dedicated to the construction of suitable Viehweg-Zuo sheaves).  
  \item The difference in the two papers can be clearly seen in the applications of the main results. For example,  our main result   is used to prove the Persistence Conjecture for varieties with a quasi-finite morphism to the stack of canonically polarized varieties  (see Theorem \ref{conjper}; this  arithmetic consequence allowed us   to extend Lawrence-Sawin's finiteness result for smooth hypersurfaces in abelian varieties over number fields to finitely generated fields (see \cite{JM}). The main result of  the more recent paper \cite{JLSZ}   is not "strong" enough to resolve the Persistence Conjecture (because the finiteness result for $1$-pointed maps in that paper inevitably comes with an exceptional locus).
  \item The main result of our recent paper \cite{JLSZ}   has different applications. It can be used to bound the dimension of moduli spaces of maps into $U$ and to prove new inheritance of hyperbolicity properties. On the other hand, the main result of  this paper can not be used to prove such a  dimension bound nor inheritance statement.  
  \end{enumerate}

\subsection{Motivation: Lang-Vojta's conjecture on hyperbolicity}\label{section:lang} 
  We were first led to investigate the pointed version of the Shafarevich conjecture by the Lang-Vojta conjectures \cite{Lang2, Vojta87} which predict that the complex-analytic hyperbolicity of a quasi-projective variety over $\mathbb{C}$ should force a plethora of arithmetic and geometric finiteness properties. The following version of the Lang-Vojta conjectures  formed the guiding thread  in our investigations on the Weak-Pointed Shafarevich Conjecture:
  
  \begin{conjecture}[Lang-Vojta]\label{conj:lv}
  Let $X$ be a quasi-projective variety over $\CC$. Then the following statements are equivalent.
  \begin{enumerate}
  \item The variety $X$ is Brody hyperbolic.
  \item Every subvariety of $X$ is of log-general type.
  \item For every smooth quasi-projective connected curve $C$ over $\CC$, every  $c\in C(\CC)$,   every     $x\in X(\CC)$, the set of morphisms $f:C\to X$ with $f(c) =x$ is finite.
  \end{enumerate}
  \end{conjecture}

  If $X$ is projective, then the equivalence of $(1)$ and $(2)$ is a consequence of     the Green--Griffiths--Lang conjecture. The fact that $(1)\implies (3)$ for $X$ projective is well-known (and due to Urata  as explained below). The conjecture that $(3)\implies (1)$ in the projective case \emph{follows} from Lang's conjectures as explained   by Demailly implicitly in his work \cite[\S2]{Demailly} (the smoothness assumptions in Demailly's paper are irrelevant; see \cite{JKa}).  On the other hand, if $X$ is quasi-projective (and not necessarily projective), then Conjecture \ref{conj:lv} is not stated explicitly anywhere in the literature.  We  attribute this conjecture to Lang-Vojta, as it is a natural extension of their conjectures to the quasi-projective setting.
    
 The history of     results on the hyperbolicity of varieties with a quasi-finite morphism to $\mathcal{M}\otimes_{\QQ} \CC$ goes back to  Kebekus-Kov\'acs and Migliorini  \cite{KebekusKovacs3, Migliorini}; see \cite{Kovacs, KovacsSubs}  for detailed discussions. 
  In fact, a variety $U$ over $\CC$ which admits a quasi-finite morphism to $\mathcal{M}\otimes_{\QQ } \CC$ is Brody hyperbolic    \cite{Deng,  VZ}.   Recently, in \cite[Corollary~C]{DLSZ} the authors also   showed that $U$ is  Borel hyperbolic in the sense of  \cite{JKuch} (i.e., every holomorphic map $S^{\an}\to U^{\an}$ with $S$ a reduced finite type scheme over $\mathbb{C}$ is algebraic).

    On the algebraic side, assuming the fibres of $V\to U$ have ample canonical bundle,  it is known  that every closed subvariety of $U$ is of log-general type by Campana-P\u{a}un's theorem \cite{CP} (see also \cite{Sch17}). Finally, assuming $U$ is smooth and the moduli map is $U\to \mathcal{M}_h\otimes_{\QQ} \CC$ is unramified, To-Yeung and Schumacher showed that  $U$ is    Kobayashi hyperbolic   \cite{Schumacher, ToYeung}.

  The restriction in   the aforementioned results (as well as our results) to families of varieties with semi-ample canonical bundle is necessary, as   moduli stacks of Fano varieties are not necessarily hyperbolic; see \cite[Section~5]{JLFano}.

\subsection{Arithmetic application} Our actual motivation for investigating finiteness of pointed families of   polarized varieties is of an arithmetic nature.  In fact, by using Theorem \ref{thm:main} we are able to verify certain expected finiteness properties for integral points on $\mathcal{M}$; see Theorem \ref{thm2} below for a precise statement.

Let $k$ be an algebraically closed field of characteristic zero.   
A  variety 
  $X$  over $k$ is said to be \emph{arithmetically hyperbolic over $k$} 
if  there is a  $\ZZ$-finitely generated subring $A\subset k$ and a finite type separated $A$-scheme $\mathcal{X}$ with $\mathcal{X}_k \cong X$ over $k$ such that, for all $\ZZ$-finitely generated subrings $ A'\subset k$ containing $A$,  the set $\mathcal{X}(A')$  of $A'$-points  on $\mathcal{X}$ is finite.  Thus, roughly speaking, a variety is arithmetically hyperbolic if it has only finitely many integral points.

In  \cite[p.~202]{Lang2} Lang asks whether finiteness statements involving  integral points over rings of integers should continue to hold over arbitrary $\mathbb{Z}$-finitely generated   subrings of $\mathbb{C}$.   This naturally leads us to the following conjecture.

\begin{conjecture}[Persistence Conjecture]\label{conjper}
Let $k\subset L$ be an extension of algebraically closed fields of characteristic zero. 
If $X$ is an arithmetically hyperbolic variety over $k$, then $X_L$ is arithmetically hyperbolic over $L$.
\end{conjecture}

  This conjecture has been investigated in \cite{vBJK, JAut, JLevin, JLitt} in the case of varieties with a quasi-finite period map, finite covers  of semi-abelian varieties, algebraically hyperbolic varieties,  Brody hyperbolic varieties, hyperbolically embeddable varieties, and surfaces  with non-zero irregularity, respectively. 
  
We use   Theorem \ref{thm:main}  to show that the Persistence Conjecture holds for base spaces of families of     polarized varieties with semi-ample canonical bundle.
 
 \begin{theorem}[Persistence Conjecture holds for $U$]\label{thm2} Let $k$ be an algebraically closed field of characteristic zero, and let $U$ be a variety over $k$ which admits a quasi-finite morphism $U\to \mathcal{M}\otimes {k}$. Let $k\subset L$ be an extension of algebraically closed fields of characteristic zero. If $U$ is arithmetically hyperbolic over $k$, then $U_L$ is arithmetically hyperbolic over $L$.
 \end{theorem}
 
 This result is crucially used in \cite{JM} to extend finiteness results for smooth hypersurfaces in abelian varieties over number fields due to Lawrence-Sawin (see \cite{LawrenceSawin}) to the setting of abelian varieties over finitely generated fields of characteristic zero.
 

\subsection{An open problem} If $X$ is a Brody hyperbolic proper variety over $\CC$, then Urata proved that, for $(C,c)$ a smooth connected pointed curve over $\CC$ and $x$ a   point of $X$,  the set of morphisms  $f:C\to X$ with $f(c) = x$ is finite; see \cite{Urata}.  It follows that, if $U$ is  a \emph{projective} variety over $\CC$ admitting a (quasi-)finite morphism to $\mathcal{M}\otimes_\QQ \CC $, then the set of pointed maps $f:(C,c)\to (U,u)$ is finite; see \cite[Example~2.3]{JLitt} for details.  Thus, for $U$ a projective variety,  the conclusion of Theorem \ref{thm:main} can be strengthened; one only needs to choose one point on the curve $C$ to obtain the finiteness of the set of pointed maps. The Lang-Vojta conjecture formulated above (Conjecture \ref{conj:lv}) actually predicts  that this one-pointed finiteness property should hold for \emph{every} quasi-projective variety admitting a quasi-finite morphism to $\mathcal{M}\otimes_\QQ \CC$:

   \begin{conjecture}[Pointed Shafarevich Conjecture]\label{conj} Let $U$ be a variety over $\CC$ which admits a quasi-finite morphism $U\to \mathcal{M}\otimes_{\QQ} {\CC}$. If $C$ is a smooth quasi-projective connected curve over $\CC$,   $c\in C(\CC)$, and   $u\in U(\CC)$, then   the set of morphisms $f:C\to U$ with $f(c)=u$ is finite.
  \end{conjecture}
  
 We stress that, as $U$ is Brody hyperbolic,  Conjecture \ref{conj} is a consequence of Conjecture \ref{conj:lv}. Also, we emphasize that Mori's bend-and-break is a crucial ingredient of Urata's proof and that it is the lack of a suitable quasi-projective analogue of Mori's bend-and-break (as discussed before) that prevents us from  extending Urata's theorem to the quasi-projective setting.

   Roughly speaking, the Pointed Shafarevich Conjecture asserts the finiteness of   families over $C$ with one ``fixed fibre'' over one ``fixed point''  of $C$. This conjecture (Conjecture \ref{conj}) is   known  if $\deg h =1$ (so that $\mathcal{M}_h$ is a connected component of the moduli stack of smooth proper curves of genus at least two). Indeed, in this case   Conjecture \ref{conj} is a direct consequence of Arakelov-Parshin's result \cite{ArakelovShaf, Parshin3}. It is also known if  $\dim \mathcal{M}_h =1$ in which case it follows from   a well-known stacky generalization of   the theorem of de Franchis.  As we already mentioned, Conjecture \ref{conj} is also known if $U$ is proper (see \cite[Example~2.3]{JLitt} for details). In the non-proper case,  one can show that Conjecture \ref{conj} holds if   $U$ is  smooth, affine, and \emph{hyperbolically embeddable}; see \cite[Theorem~1.4]{JLevin}.   Finally, if $U$ admits a quasi-finite period map (e.g., the   fibres of the family $V\to U$ satisfy infinitesimal Torelli), then Conjecture \ref{conj} is a consequence of Deligne's finiteness theorem for monodromy representations \cite{DeligneMonodromy} and the Rigidity Theorem in Hodge theory  \cite[7.24]{Schmid}; see  \cite[Theorem~1.5]{JLitt} for a detailed proof.  
 
\subsection{A remark on stackyness}
We restrict ourselves to varieties $U$ with a quasi-finite morphism to the stack $\mathcal{M}$ and do not work directly with the stack $\mathcal{M}$ itself only to avoid technicalities such as resolving singularities of the stack $\mathcal{M}$, the theory of Higgs bundles on stacks, and variations of Hodge structure on stacks.   

Another reason for not working with the stack directly is that  a simple argument (similar to Chevalley-Weil's theorem) shows that our main result can be \emph{used} to prove that the stack $\mathcal{M}$ satisfies similar finiteness properties, under the assumption that  $\mathcal{M}$ is \emph{uniformizable}, i.e., there is a scheme $M$ and a finite \'etale surjective morphism $M\to \mathcal{M}$.  Examples of components of $\mathcal{M}$ which are uniformizable are  the stack $\mathcal{C}_g$ of curves of genus $g\geq 2$, as well as the stack of smooth  hypersurfaces in projective space \cite{JLlevel} and the stack of very canonically polarized varieties \cite{PoppIII}.  In light of these examples, it seems reasonable to suspect that the stack $\mathcal{M}\otimes_{\QQ} \CC$   is in fact uniformizable. We   formulate our expectation  as a conjecture:

 \begin{conjecture}[Uniformizability of $\mathcal{M}$]
The stack $\mathcal{M}_{h}$ over $\QQ$ of   polarized smooth proper geometrically connected varieties  with semi-ample canonical bundle is uniformizable.
 \end{conjecture}

 \begin{ack} The first named author gratefully acknowledges support of  the IHES.
 The authors gratefully acknowledge support from SFB/Transregio 45. We are most grateful to the referees for  many helpful comments. 
 \end{ack}
 
 \begin{con} A variety over $k$ is an integral quasi-projective scheme over $k$.
 A pair $(X,L)$ is a (smooth projective connected) polarized variety over $k$ if $X$ is a smooth projective connected variety over $k$ and $L$ is an ample line bundle on $X$.
 
 If $L/k$ is an extension of fields and $U$ is a variety (or  algebraic stack) over $k$, then we let $U_L $ denote the base-change of $U$ to $\Spec L$. We will  sometimes denote $U_L$ by $U\otimes_k L$ to avoid  confusion.
 \end{con}

\section{The moduli space of pointed maps}
If $A$ and $B$ are projective schemes over $k$, as is explained in \cite[\S 2.1]{DebarreBook},  the functor \[ \mathrm{Sch}/k^{\textrm{op}} \to \mathrm{Sets}, \quad T\mapsto \Hom_T(A_T, B_T)\] is representable by a scheme $\underline{\Hom}_k(A,B)$. This scheme is a countable union \[\underline{\Hom}_k(A,B) = \sqcup_{P\in \QQ[t]} \underline{\Hom}_k^P(A,B)\] of quasi-projective schemes $\underline{\Hom}_k^P(A,B)$, where  $P$ runs over all Hilbert polynomials (computed with respect to a fixed choice of ample line bundle on $A$ and $B$, respectively).

Let $C$ be a smooth quasi-projective curve over $k$ with smooth projective model $\overline{C}$. Let $U$ be a dense open of a projective variety $X$ over $k$.  Write $S:=X\setminus U$. We consider the set $\Hom(C,U)$ of morphisms $f:C\to U$ as a subset of the set of  $k$-points of the   scheme $\underline{\Hom}_k(\overline{C}, X)$.   The following well-known structure result is proven   in  \cite[5.1~Definition-Lemma]{keelmckernan} (see also \cite[Corollary~3.11]{BJR}). 

\begin{proposition}[Scheme structure] The subset   $\Hom(C,U)$ is locally closed in  $\underline{\Hom}_k(\overline{C},X) $.  In fact,  $\underline{\Hom}_k(\overline{C},S)$ is a closed subscheme of $\underline{\Hom}_k(\overline{C},X)$   and $\Hom(C,U)$ is closed in the open subscheme  $\underline{\Hom}_k(\overline{C},X)\setminus \underline{\Hom}_k(\overline{C},S)$ of $\underline{\Hom}(\overline{C},X)$.
\end{proposition}

We let $\underline{\Hom}(C,U)$ be the locally closed reduced subscheme of $\underline{\Hom}_k(\overline{C},X)$ whose $k$-points are the morphisms $f:\overline{C}\to X$ with $f(C)\subset U$.  We stress that this scheme might not represent the functor $T\mapsto \Hom(C_T,U_T)$.

 \subsection{Boundedness}
Our finiteness result relies on  the following boundedness result of
Ascher-Taji \cite{AT}.

\begin{theorem}[Boundedness]\label{thm:boundedness} If $U$ is a variety which admits a quasi-finite morphism to $\mathcal{M}\otimes_{\mathbb{Q}} k$ and  $C$ is  a    smooth quasi-projective curve, then the  scheme $\underline{\Hom}(C,U)$  is of finite type over $k$. 
\end{theorem}

We   note that weaker (but related) results were also obtained     by Deng--Lu--Sun--Zuo  \cite[Theorem~D]{DLSZ} and Viehweg--Zuo \cite[Theorem~6.2]{VZ02}.

We stress that Theorem \ref{thm:boundedness} is   new in the setting of polarized varieties with non-ample  canonical bundle. Indeed, if $U$ parametrizes a family of canonically polarized varieties with quasi-finite moduli map, then Theorem \ref{thm:boundedness} follows from the boundedness theorem of Kov\'acs-Lieblich \cite{KovacsLieblich}.

\subsection{Infinitesimal deformations} 
Given a morphism of smooth projective varieties $f:X\to Y$, the space of infinitesimal deformations of $f$ (keeping $X$ and $Y$ fixed) is given by $\mathrm{H}^0(X, f^\ast T_Y)$ (see \cite[\S 2.3]{DebarreBook}).   If  in addition  $B\subset X$ is a finite closed subset and X is one-dimensional, then the space of infinitesimal deformations of $f$ for which $f|_B$ does not change is given by  
$
\mathrm{H}^0(X,f^*T_Y \otimes \mathcal{O}_C(-B)).$  We will need a  generalization of this fact to the quasi-projective setting.

\begin{proposition}[Infinitesimal deformations]\label{log_deform}
Let $C$ be a smooth quasi-projective curve with smooth projective model $\overline{C}$, let $D=  \overline{C}\setminus C$,   let $B\subset C$ be a finite set of closed points, let $Y$ be a smooth projective variety and let $S$ be a simple normal crossings divisor on $Y$. 

Let $f:\overline{C}\to Y$ be a morphism  with $f^{-1}(S) = D$. The space  of infinitesimal deformations $g$ of $f$ satisfying $g^{-1}(S) = D$ and $g|_B = f|_B$   is    the subspace $ \mathrm{H}^0(C,f^*T_Y(-\log S) \otimes \mathcal{O}_C(-B))$ 
  of $\mathrm{H}^0(\overline{C}, {f}^\ast T_Y)$.
\end{proposition}

\begin{proof}
We adapt  the proof of \cite[Proposition~5.3]{keelmckernan}   to our setting.

Write $S= \sum^q_{i=1}S_i$ with $S_i$ the irreducible components of $S$.  
By \cite[Lemma~5.2]{keelmckernan},  we have the following short exact sequences  
  \begin{align}
    \label{sxs_def_1}
    0 \to T_Y(-\log S) \to T_Y \to \bigoplus^q_{i=1}\mathcal{O}_{S_i}(S_i) \to 0
  \end{align}
  and
  \begin{align}
    \label{sxs_def_2}
    0 \to \mathcal{O}_{S_i} \to T_Y(-\log S)|_{S_i} \to T_{S_i}(-\log [S_i \cap (S-S_i)]) \to 0
  \end{align}


  Define $D_i:= f^*S_i$. 
By \cite[\S 2.9]{DebarreBook},  the  space of infinitesimal deformations $g$ of $f$ with $f
_B = g|_B$ is given by $\mathrm{H}^0(\overline{C}, f^*T_Y \otimes \mathcal{O}_C(-B))$. On the other hand, using \eqref{sxs_def_1}, \eqref{sxs_def_2}, and the fact that $B \subset C$, we obtain the commutative diagram
\[
  \xymatrix{
  0 \ar[r] & \mathrm{H}^0(\overline{C}, f^*T_Y(-\log S_i) \otimes \mathcal{O}_C(-B)) \ar[r] \ar[d] & \mathrm{H}^0(\overline{C}, f^*T_Y \otimes \mathcal{O}_{\overline{C}}(-B)) \ar[r] \ar[d]^r & \mathrm{H}^0(\overline{C},f^*\mathcal{O}_{S_i}(S_i)) \ar@{=}[d] \\
  0 \ar[r] & \mathrm{H}^0(\overline{C}, (f|_{D_i})^*T_{S_i}) \ar[r] & \mathrm{H}^0(\overline{C}, (f|_{D_i})^*T_Y) \ar[r]^-z  &  \mathrm{H}^0(\overline{C}, (f|_{D_i})^*\mathcal{O}_{S_i}(S_i))
  }
\]

Note that the  space of infinitesimal deformations $g$ of $f$ which satisfy $g^{-1}(S_i) = D_i$ and $f|_B = g|_B$     is the kernel of the composition $z \circ r$ and thus by the diagram equal to $\mathrm{H}^0(\overline{C}, f^*T_Y(-\log S_i) \otimes \mathcal{O}_C(-B))$. Therefore,   the space  of infinitesimal deformations $g$ of $f$ satisfying $g^{-1}(S) = D$ and $g|_B = f|_B$   equals  
\begin{eqnarray*}
 \bigcap^q_{i=1}\mathrm{H}^0(f^*T_Y(-\log S_i) \otimes \mathcal{O}_C(-B))   &=& \mathrm{H}^0(C,f^*T_Y(-\log S) \otimes \mathcal{O}_C(-B)). 
\end{eqnarray*} This concludes the proof.
\end{proof}

\section{Deforming pointed families of polarized varieties} \label{section:3}
This section is devoted to proving the rigidity of polarized families of   with enough fixed fibres.  
More precisely,   we  study non-trivial deformations of   pointed curves in the moduli stack of polarized manifolds with semi-ample canonical bundle using Viehweg-Zuo sheaves. Notably, we   use negativity results of kernels of Higgs bundles to force the rigidity of families  of polarized varieties, assuming ``enough'' base points have been chosen on the curve;  see Theorem \ref{thm:pointed_rigidity} for a precise statement. Before stating the main result of this section, we introduce the necessary notation.

\subsection{Starting point: a non-rigid pointed morphism}  \label{section:notation}
Recall that  $\mathcal{M}_h$ denotes the stack of smooth proper polarized varieties $(X,L)$ whose  Hilbert polynomial (with respect to $L$) is $h$ and whose geometric fibres have semi-ample canonical bundle. 
Let $U$ be a  quasi-projective variety over $\CC$,  and let $\varphi:U\to \mathcal{M}_{h}\otimes \CC$ be a quasi-finite morphism of stacks. 

The pull-back of the universal family over $\mathcal{M}_{h}\otimes \CC$ along $\varphi$ induces a polarized variety over $U$. Concretely: there is  
 a smooth proper morphism  $\pi:V \to U$ whose geometric fibres are connected with semi-ample canonical divisor, and a  $\pi$-ample line bundle $L$ on $V$, so that  the moduli map $U\to \mathcal{M}_h\otimes \CC$ associated to the polarized variety $(V,L)\to U$  is given by $\varphi$.   We let $d:=\deg h$ denote the fibre dimension. 
 
Let $Y$ be an integral projective compactification of $U$ with $S:=Y\setminus U$ a divisor.  
Let $\psi:\hat{Y}\to Y$ be a proper birational surjective morphism with $\hat{Y}$ a smooth projective variety    such that the following statements hold:
\begin{enumerate}
\item Let $\widehat{U}$ be the inverse image of $U$ in $\widehat{Y}$. Let $\widehat{V}$ be the inverse image of $\widehat{U}$ in $\widehat{X}$. Then $\widehat{V}\to \widehat{U}$ is a smooth proper morphism. 
\item The subvariety $\hat{S} = \psi^{-1}(S)$ is a simple normal crossings divisor.
\item Let  $\mathrm{Exc}(\psi)\subset Y$ be the image of the exceptional locus of $\psi$. Then the inverse image $E$ of $\mathrm{Exc}(\psi)$ along $\psi:\hat{Y}\to Y$ is a simple normal crossings divisor and $\hat{S}\cup E$ is a simple normal crossings divisor.
\end{enumerate}

 We choose a compactification $X\to Y$ of the composed map $V\to U\subset Y$,     pull-back the family $X\to Y$ along $\psi:\hat{Y}\to Y$, and then resolve the singularities of the total space   to obtain  a  surjective proper morphism $g:\hat{X}\to \hat{Y}$ of smooth projective varieties  with $\hat{\Delta}:= g^{-1}(\hat{S})$ a normal crossings divisor  on $\hat{X}$, so that the following diagram commutes
\[
\xymatrix{V \ar[d]_{\pi} \ar[rr]^{\textrm{open immersion}} & &  {X} \ar[d]   & &\hat{X}  \ar[ll] \ar[d]^{g} \\ U \ar[rr]_{\textrm{open immersion}}  & &  {Y}  & & \hat{Y}.\ar[ll]^{\psi} }
\]
We define $$F^{d-1,1}:= \left(R^1g_*T_{\hat{X}/\hat{Y}}(- \log \hat{\Delta}) \right)/ \textrm{torsion}.$$ 
This will be a  component of a graded  Higgs bundle constructed below in Section \ref{section:graded}.

We let  $T_g\subset  R^1g_*T_{\hat{X}/\hat{Y}}(- \log \hat{\Delta})$ be the torsion in the coherent sheaf $ R^1g_*T_{\hat{X}/\hat{Y}}(- \log \hat{\Delta})$. We let  $\mathrm{Supp}(T_g)$ be its support, and note that it is a proper closed subset of  $\hat{Y}$.

  Let $C$ be a smooth projective curve, let $D\subset C$ be a finite subset,  and let $B\subset C\setminus D$ be a finite closed subset.  We will establish a rigidity result for morphisms $C\to U$ which restrict to a fixed map on $B$, under a suitable assumption on the cardinality of $B$.  We will argue by contradiction and therefore assume the existence of a family of morphisms.  Let us be more precise.
  
  We \textbf{assume} that there is a smooth affine curve $T$  and a morphism $\hat{f}:C\times T\to \hat{Y}$ such that the following properties hold.    
\begin{itemize}
\item The image of $f:=\psi\circ \hat{f}:C\times T\to Y$ in $Y$ is two-dimensional (i.e., $f$ is generically finite onto its image),
\item  for every $t\in T(\mathbb{C})$, the morphism $f_t:= f|_{C\times \{t\}}:C\to Y$ is   non-constant    with $f_t^{-1}(S) = D$,
\item  for every $t\in T(\mathbb{C})$, the image of $f_t$ is  not contained in  $\mathrm{Exc}(\psi) \cup \mathrm{Supp}(T_g)$, and
\item    for every $s$ and $t$ in $T(\mathbb{C})$, we have that  $(f_s)|_B = (f_t)|_B$.   
\end{itemize}

The main result of this section (Theorem \ref{thm:pointed_rigidity})  is that the existence of the data above implies  
\[ 
\# B < \frac{d-1}{2}\Big(2g(C)-2 + \# D\Big).
\] 

Before we proceed with the proof, we fix the following notation.  

 First, we will fix throughout  a point $t_0\in T(\mathbb{C})$.  Define $f_0 = f|_{C\times \{t_0\}}$ and $\hat{f_0} = \hat{f}|_{C\times \{t_0\}}$.  If $t\in T(\mathbb{C})$, we define $f_t = f|_{C\times \{t\}}$ and $\hat{f}_t := \hat{f}|_{C\times \{t\}}$.  It will be helpful to view all $f_t$ (with $t\in T(\mathbb{C})$) as deformations of $f_0$.  Most of our arguments below   hold for all $t$ in $T(\mathbb{C})$, and we will specialize to the case $t=t_0$ only in specific situations (e.g., Definition \ref{def:xi}).

Define $B_1$ to be the subset of points $b$ in $B$ such that $f_0(b) \in \mathrm{Exc}(\psi)$, and let $B_2 = B\setminus B_1$. Since  $\psi$ is an isomorphism over $Y\setminus \mathrm{Exc}(\psi)$, we may consider  $f_0(B_2)\subset   Y\setminus \mathrm{Exc}(\psi)\subset Y$  as a subset of $\hat{Y}$.     Note that    $\hat{f_0}$ is non-rigid as an element of the space   $$\mathrm{Hom}\left( (C, D+B_1), (\hat{Y},\hat{S}+E); B_2\right)$$ of morphisms $f:C\to \hat{Y}$ with $f^{-1}(\hat{S}+E) = D \cup B_1$ and $f|_{B_2}$ fixed.

  \subsection{The construction of the map $\xi$}\label{section:xi}
In what follows, we will   use the  data above to construct, for $t$ a   point of $T(\mathbb{C})$, a non-zero map
$$
\xi: \mathcal{O}_C(B) \to \hat{f}^*_tF^{d-1,1}.$$ 
We define $E_0:=\psi^{-1}(f_0(B_1))\subset \hat{Y}$.   By definition, as $\hat{f}$ deforms $\hat{f}_0$ inside   $\mathrm{Hom}\left( (C, D+B_1), (\hat{Y},\hat{S}+E); B_2\right)$  and  $(f_t)|_B = (f_0)|_B$ for every $t \in T(\mathbb{C})$, we obtain 
  the following commutative diagram:  
\[
\xymatrix{
T \times B_1 \ar[r] \ar@{^{(}->}[d] & E_0 \ar@{^{(}->}[d] \\
T \times C   \ar[r]^-{\hat{f}}        & \hat{Y}.
}
\]
Consider the Kodaira-Spencer map (composed with the quotient map)
\[
\tau^{d,0}:\, T_{\hat{Y}}(-\log \hat{S}) \to F^{d-1,1}.
\]

Since the pull-back family $g:\hat{X} \to \hat{Y}$ is trivial along  $E_0 \subset \hat{Y}$,   the  composed map
\[
\mathcal{O}_{E_0} \xrightarrow{\tilde{\tau}^{d,0}|_{E_0}} (F^{d-1,1} \otimes \Omega^1_{\hat{Y}}(\log \hat{S}))|_{E_0} \to F^{d-1,1}|_{E_0} \otimes \Omega^1_{E_0}
\]
is zero, where $\tilde{\tau}^{d,0}$ is the map induced by the Kodaira-Spencer map $\tau^{d,0}$.

Thus, as  the following diagram
\begin{equation}\label{ks}
  \xymatrix{
    &  (\hat{f}|_{T \times B_1})^*(F^{d-1,1}|_{E_0} \otimes \Omega^1_{E_0}) \ar[r] &  \hat{f}^*F^{d-1,1}|_{T \times B_1} \otimes \Omega^1_{T \times B_1} \\
    \mathcal{O}_{T \times B_1} \ar[r] \ar[ur]^{\equiv 0} & \hat{f}^*(F^{d-1,1} \otimes \Omega^1_{\hat{Y}}(\log \hat{S}))|_{T \times B_1} \ar[r] \ar[u] & \hat{f}^*F^{d-1,1}|_{T \times B_1} \otimes \Omega^1_{T \times C}(\log (T \times D))|_{T \times B_1} \ar[u]
      }
\end{equation} commutes, it follows that  the restricted Kodaira-Spencer map $T_{T \times B_1} \to \hat{f}^*F^{d-1,1}|_{T \times B_1}$ is zero.

Note that the log tangent bundle of the product surface decomposes as the direct sum
\[
T_{T \times C}(-\log (T \times D)) = p^*_1T_T \oplus \left(p^*_2T_C(-\log D)\right).
\]
Moreover, note that \[p^*_1T_T|_{T \times B_1} = T_{T \times B_1}.\] Since the family over $T\times (C\setminus D) $ has maximal variation in moduli (i.e., two-dimensional image in $\hat{Y}$),    the composed map 
\[
\mathcal{O}_C \cong p^*_1T_T|_{\{t\} \times C} \hookrightarrow  \hat{f}^*_tT_{T \times C}(-\log (T \times D)) \xrightarrow{\mathrm{d} \hat{f}} \hat{f}^*_tT_{\hat{Y}}(-\log \hat{S}) \xrightarrow{\hat{f}^*_t\tau^{d,0}} \hat{f}^*_t F^{d-1,1}
\] 
is non-zero.  (Here we use that the Kodaira-Spencer map is generically injective.) Moreover, 
since the  deformation $\hat{f}:\, T \times C \to \hat{Y}$ of $f_0$ maps $B_2$ to a fixed subset of $Y\setminus \mathrm{Exc}(\psi)\subset \hat{Y}$, it follows from  Proposition~\ref{log_deform} that  the composition of the first two maps $\mathcal{O}_C \to \hat{f}^*_tT_{\hat{Y}}(-\log \hat{S})$   factors through
\[
\mathcal{O}_C \to \hat{f}^*_tT_{\hat{Y}}(-\log \hat{S})(-B_2).
\]
 
\begin{proposition}\label{prop:xi}
The composed map  \[\mathcal{O}_C \to \hat{f}^*_tT_{\hat{Y}}(-\log \hat{S})(-B_2) \to \hat{f}^*_tF^{d-1,1}(-B_2)\] factors through a map \[\mathcal{O}_C \to \hat{f}^*_tF^{d-1,1} \otimes \mathcal{O}_C(-B).\]
\end{proposition}
\begin{proof}
 By Diagram~(\ref{ks}),  the composed map 
\[
\mathcal{O}_C \to \hat{f}^*_tT_{\hat{Y}}(-\log \hat{S})(-B_2) \to \hat{f}^*_tF^{d-1,1}(-B_2)\to \hat{f}^*_tF^{d-1,1}|_{B_1}\]   
is the zero map. Thus, from the following diagram  
\[
\xymatrix{
0 \ar[r] & \mathcal{O}_C(-B_1) \ar[r] \ar[d] & \mathcal{O}_C \ar[r] \ar[d] \ar@{-->}[ld]_{\exists !}  & \mathcal{O}_{B_1}  \ar[r] \ar[d]^{\equiv 0} &  0 \\
0 \ar[r] & \hat{f}^*_tF^{d-1,1} \otimes \mathcal{O}_C(-B) \ar[r] & \hat{f}^*_tF^{d-1,1}(-B_2) \ar[r]    & \hat{f}^*_tF^{d-1,1}|_{B_1}  \ar[r] &  0
}
\]
we deduce that the   map  $\mathcal{O}_C \to \hat{f}^*_tF^{d-1,1}(-B_2)$ factors through \[\mathcal{O}_C \to \hat{f}^*_tF^{d-1,1} \otimes \mathcal{O}_C(-B).\]
 This concludes the proof.
\end{proof}
\begin{definition}\label{def:xi} For every $t$ in $T(\mathbb{C})$, 
we define   \[ 
\xi: \mathcal{O}_C(B) \to \hat{f}^*_tF^{d-1,1}.
\]  to be the map obtained after tensoring $\mathcal{O}_C \to \hat{f}^*_tF^{d-1,1} \otimes \mathcal{O}_C(-B)$  in Proposition \ref{prop:xi} with $\mathcal{O}_C(B)$. 
\end{definition}  

\begin{remark}\label{remark:xi}
If  $Y $ is smooth, then   the construction of $\xi$ is straightforward. Indeed, in this case, the map $\xi$ is  obtained by composing the  map $\mathcal{O}_C(B)\to f_0^\ast T_Y(-\log S)$ (induced by the non-rigidity of $f_0$) with the Kodaira-Spencer map. The general construction of $\xi$ is slightly more complicated due to the possible singularities of $U$.  
\end{remark}

The map $\xi$   is crucial in our proof of Theorem \ref{thm:pointed_rigidity} below. In fact, in the proof of Theorem \ref{thm:pointed_rigidity} below, we will combine the existence of the  map $\xi$   with Viehweg-Zuo's construction of certain graded Higgs bundles to prove an upper bound on $\# B$.
 
 \subsection{Viehweg-Zuo's graded Higgs bundle}\label{section:graded}
 Let $T^q_{\hat{X}/\hat{Y}}(-\log \hat{\Delta})$ denote the $q$-th wedge product of the relative log tangent sheaf $T_{\hat{X}/\hat{Y}}(-\log \hat{\Delta})$.
 As we alluded to above, the    sheaf $F^{d-1,1}$ is a piece of a graded Higgs bundle constructed by Viehweg-Zuo. More precisely, 
by Viehweg-Zuo's construction   \cite[\S 4]{VZ02}, there is a graded Higgs bundle $(F,\tau)$ associated to the family $g:\, \hat{X} \to \hat{Y}$ with bi-graded structure $(\bigoplus_{p+q =d} F^{p,q}, \bigoplus_{p+q =d} \tau^{p,q})$, where
\[
F^{n-q,q} = F^{p,q} := \left(R^qg_* T^q_{\hat{X}/\hat{Y}}(-\log \hat{\Delta})\right) / {\mathrm{torsion}}.
\] 
and  the morphism $\tau^{p,q}$ is induced by the edge morphism of a long exact sequence of higher direct image sheaves
\[
\tau^{p,q} :\, R^qg_* T^q_{\hat{X}/\hat{Y}}(-\log \hat{\Delta}) \to R^{q+1}g_* T^{q+1}_{\hat{X}/\hat{Y}}(-\log \hat{\Delta}) \otimes \Omega^1_{\hat{Y}}(\log \hat{S}).
\] 

 By construction, the first Higgs map $$\tau^{d,0}:\, F^{d,0} \cong \mathcal{O}_{\hat{Y}} \to F^{d-1,1} \otimes \Omega^1_{\hat{Y}}(\log \hat{S})$$ is   the Kodaira-Spencer class (inducing the Kodaira-Spencer map above).  
 Let $\tau_C$ be the composed morphism \[\hat{f}^*_{t}F \xrightarrow{\hat{f}^*_t\tau} \hat{f}^*_{t}F \otimes \hat{f}^*_t\Omega^1_{\hat{Y}}(\log \hat{S}) \xrightarrow{\mathrm{Id} \otimes \mathrm{d}\hat{f}_t} \hat{f}^*_{t}F \otimes \Omega^1_C(\log D).\]
 Composing the non-zero map $\xi$ with   $\tau_C^{d-1,1}$,  we obtain the following sequence of maps
\begin{eqnarray}\label{eqn}
\mathcal{O}_C(B) \xrightarrow{\xi}  \hat{f}^*_tF^{d-1,1} \xrightarrow{\tau^{d-1,1}_C} \hat{f}^*_{t}F^{d-2,2} \otimes \Omega^1_C(\log D).
\end{eqnarray}
Tensoring the composed map $(\ref{eqn})$ above with $T_C(-\log D)$, one obtains
\[
\mathcal{O}_C(B) \otimes T_C(-\log D) \to  \hat{f}^*_tF^{d-1,1} \otimes T_C(-\log D) \to \hat{f}^*_tF^{d-2,2}
\]
 
 Recall that $t$ in $T(\mathbb{C})$ was any point. In particular, we can take $t$ to be $t_0$ (the point fixed in Section \ref{section:notation}).
\begin{definition}\label{definition:iterated}
 For every $m\geq 0$,   by iterating the above construction, we obtain the following map of sheaves
\[
\tau^m:\, \mathcal{O}_C(B) \otimes T_C(-\log D)^{\otimes m} \longrightarrow f^*_0F^{d-m-1,m+1}.
\qedhere \]
\end{definition}
The maps $\tau^m$ are ``pointed'' variants of iterated Higgs maps; see   \cite{KovacsStrong} and \cite{ViehwegZuoCM2} for the classical iterated Higgs maps, respectively.

 \begin{remark}[A brief history of Viehweg-Zuo sheaves]
In a series of papers \cite{VZ01, VZ02, VZ}, Viehweg-Zuo proved the negativity of the kernel of the Higgs map $\tau$ and used this negativity to derive the Brody hyperbolicity of the base space $U$, assuming the fibres of $V\to U$ have ample canonical bundle (see \cite{Deng} for the general case). The idea is to construct a new family of varieties such that there is a natural comparison map from the Higgs bundle $(F,\tau)$ to the Hodge bundle $(E,\theta)$ of the new family twisted with a certain anti-ample line bundle $A^{-1}$, so that,  as a consequence, we get a nonzero map from $\mathrm{Ker}(\tau^{d-m,m})$ to $A^{-1} \otimes \mathrm{Ker}(\theta^{d-m,m})$. Then,  as the Griffiths curvature formula for the  Hodge bundle implies that the kernel of $\theta$ is semi-negative, it follows that $\mathrm{Ker}(\tau^{d-m,m})$ is strictly negative, so that the dual of $\mathrm{Ker}(\tau^{d-m,m})$ induces a big subsheaf of the $m$th-symmetric power of the cotangent sheaf of the base space; the latter subsheaf  is commonly referred to as  the \emph{Viehweg-Zuo subsheaf}, and its existence has proven to be very useful in proving the hyperbolicity of moduli spaces of polarized varieties. 
\end{remark}

\begin{remark}[Further applications of Viehweg-Zuo sheaves]
  Campana-P\u{a}un \cite{CP} used  the existence of Viehweg-Zuo sheaves to prove    Viehweg's hyperbolicity conjecture; see  \cite{CKT-16, Sch17} for  a detailed discussion. Also, Deng  used properties of Viehweg-Zuo sheaves to prove the hyperbolicity of the moduli space of polarized varieties with big and semi-ample canonical bundle \cite{Deng}. Moreover, the properties of Viehweg-Zuo sheaves can be used to prove rigidity of certain families of canonically polarized varieties with maximal variation in moduli \cite{ViehwegZuoDisc}.  Finally, let us mention that Popa-Schnell used Hodge modules to construct Viehweg-Zuo sheaves; see \cite{Popa-Sch}.
\end{remark}
%

 \subsection{The upper bound on $\# B$}
 To prove our main result that $\#B < \frac{d-1}{2}\Big(2g(C)-2 + \# D\Big)$, we will use the following  proposition. (The proof of this builds on techniques employed in  \cite{VZ02}.)
 
 \begin{proposition}\label{prop:vz2} Let $t\in T(\mathbb{C})$.  Then, 
 there is   an ample line bundle $A$ on $\widehat{Y}$,  a positive integer $\nu$, a ramified covering $\phi:\, Y' \to \hat{Y}$  whose branch locus does not contain the image of $f_t\colon C\to Y$, and  a modification $g':\,  {X}' \to Y'$ of   the base change of the family $g:\hat{X}\to\hat{Y}$ along $\phi:Y'\to Y$ such that $\phi^*A = A'^{\nu}$ for some ample line bundle $A'$ on $Y'$   and  such that $(\Phi^*\Omega^d_{\hat{X}/\hat{Y}}(\log \hat{\Delta}) \otimes g'^*A'^{-1})^{\otimes \nu}$ is globally generated,  where  $\Phi\colon X'\to \widehat{X}$ is  $\phi\circ g$.
 \end{proposition}
 \begin{proof}
 The idea is to use   Kawamata's covering trick. More precisely,  using that the fibres of $\pi:V\to U$ have semi-ample canonical bundle,  Viehweg-Zuo proved in {\cite[Corollary~3.6]{VZ02}} that, for $\nu >0$ large enough, the sheaf \[g_*(\Omega^d_{\hat{X}/\hat{Y}}(\log \hat{\Delta}))^{\nu}\] is big in the sense of Viehweg; see \cite[Definition~1.1]{VZ02} for a precise definition. By  \cite[Lemma~1.2]{VZ02}, there is   an ample line bundle $A$ and a positive integer $\nu$ such that 
  the sheaf \[(\Omega^d_{\hat{X}/\hat{Y}}(\log \hat{\Delta}))^{\nu} \otimes g^*A^{-1}\] is globally generated.   Then, there is a ramified covering $\phi:\, Y' \to \hat{Y}$ whose branch locus does not contain the image of $C\to Y$ such that $\phi^*A = A'^{\nu}$ for some ample line bundle $A'$ on $Y'$ (see \cite[Lemma~2.3.(a)]{VZ02}). In this way, fixing  $g':\,  {X}' \to Y'$ to be a modification of the base change of the family $g:\hat{X}\to\hat{Y}$ along $\phi:Y'\to Y$, we get sections of $(\Phi^*\Omega^d_{\hat{X}/\hat{Y}}(\log \hat{\Delta}) \otimes g'^*A'^{-1})^{\otimes \nu}$, so that it is globally generated.  
 \end{proof}
 
 We are now ready to prove the main result of this section.  
 
\begin{theorem}[Pointed Rigidity]\label{thm:pointed_rigidity}  The inequality 
\[ 
\# B < \frac{d-1}{2}\Big(2g(C)-2 + \# D\Big)
\] holds. 
\end{theorem}
\begin{proof} 

We choose  $A$, $\nu$, $\phi:Y'\to \hat{Y}$, $g':X'\to Y'$  and $A'$ as in  Proposition \ref{prop:vz2}. Then, following the arguments in \cite[Lemma~4.4]{VZ02},  we    obtain the  Hodge bundle $(E_{Y'},\theta_{Y'})$ over $Y'$ and the following commutative diagram  
\begin{align}\label{comm-diag}
\xymatrix{
\phi^*F^{p,q} \ar[rr]^-{\phi^\ast \tau^{p,q}} \ar[d]^{\rho^{p,q}} & & \phi^*F^{p-1,q+1} \otimes \phi^* \Omega^1_{\hat{Y}}(\log \hat{S}) \ar[d]^{\rho^{p-1,q+1} \otimes \iota}\\
A'^{-1} \otimes E^{p,q}_{Y'} \ar[rr]^-{\mathrm{Id}_{(A')^{-1}}\otimes\theta^{p,q}_{Y'}} & &  A'^{-1} \otimes E^{p-1,q+1}_{Y'} \otimes \Omega^1_{Y'}(\log (S' + S^{new})).
}
\end{align}
Here $S'$ is the base change of $\hat{S}$ to $Y'$, and $S^{new}$ is the additional discriminant locus of the new family (i.e., the singular fibres of the new family lie over $S'\cup S^{new}$). A detailed construction of the diagram above can be found in the proofs of Lemma~4.1 and Lemma~4.4.(i) of \cite{VZ02}.  By using Kawamata's covering trick again (\emph{cf.} \cite[Theorem~17]{KawamataChar}), we may and do assume that $(E_{Y'},\theta_{Y'})$ has unipotent local monodromy along the boundary. 
 
We now pull-back $\hat{f}_0:C\to \hat{Y}$ along $Y'\to \hat{Y}$, and let $\phi_C\colon C'\to C$   denote  the induced ramified covering of $C$.  Then,  we can construct  $(E_{C'},\theta_{C'})$ on $C'$ from the Hodge bundle $(E_{Y'},\theta_{Y'})$  similarly to the construction of $F_C:= \hat{f}^*_0(F,\tau_C)$. Denote by $A'_{C'}$ the pull back of the ample line bundle over $C'$.
Using the comparison maps we obtain a   map
\[
\phi_C^*F^{d-m-1,m+1}_C \to A'^{-1}_{C'} \otimes E^{d-m-1,m+1}_{C'}
\]
for each non-negative integer $m$.
 Moreover, for each such $m$, we also have the  (pointed) iterated Higgs map (see Definition \ref{definition:iterated})
 \[\tau^m : \mathcal{O}_C(B) \otimes T_C(-\log D)^{\otimes m} \to  F^{d-m-1,m+1}_C.\]  Base-changing  $\tau^m$ to $C'$, we get the composed map
\[
\zeta_m:\, \phi_C^*\mathcal{O}_C(B) \otimes \phi_C^*T_C(-\log D)^{\otimes m} \to \phi^*F^{d-m-1,m+1}_C \to A'^{-1}_{C'} \otimes E^{d-m-1,m+1}_{C'}.
\]
Let  $m\geq 0$ be such that the map $\zeta_m$ is  nonzero   and the composition $\theta^{d-m-1,m+1}_{C'} \circ \zeta_m \equiv 0$. (One may refer to this integer $m$ as the  maximal length of the (pointed) iterated Higgs map $\zeta_m$.)

To see that such an integer $m$ exists, it suffices to note that  the following composed map
\[
\zeta_0 :\,\phi_C^*\mathcal{O}_C(B) \xrightarrow{\phi^*\xi} \phi_C^*F^{d-1,1}_C \xrightarrow{\rho^{d-1,1}} A'^{-1}_{C'} \otimes E^{d-1,1}_{C'}
\]
is nonzero. 
Indeed, the fact that $\zeta_0$ is non-zero follows from the generic injectivity of $\phi^*\xi$ and $\rho^{d-1,1}$. The first map is the Kodaira-Spencer map and thus it is generically injective because of the non-isotriviality of the pull-back family. The generic injectivity of $\rho^{d-1,1}$ follows from Deng's version of the generic Torelli theorem for Viehweg-Zuo Higgs bundles (see  \cite[Theorem~D]{Deng}), as  the induced polarized family over $T\times (C\setminus D)$ has maximal variation in moduli.
In this way, we obtain that $0 \leq m \leq d-1$, where the upper bound $m\leq d-1$ follows from   the fact that $\theta^{0,d}_{C'} \equiv 0$.


 We shall use the iterated (composed) Higgs maps $\zeta_0, \zeta_1, \dots, \zeta_m$ to construct a sub-Higgs bundle $(G,\tau_G)$ of $A'^{-1}_{C'} \otimes E_{C'}$, so that we get a sub-Higgs bundle $(A'_{C'},0) \otimes (G,\tau_G)$ of the Hodge bundle $(E_{C'},\theta_{C'})$. Write $G^{d,0}:= \{0\}$ and $G^{d-j,j} = \mathrm{Im}(\zeta_{j-1})$ for $j=1,2,\dots,m+1$. Note that, for every $1\leq j \leq m+1$, the sheaf $G^{d-j,j}$ is a subsheaf  of $A'^{-1}_{C'} \otimes E_{C'}$. Moreover, it follows from the commutativity of  Diagram (\ref{comm-diag}) that
\[
\theta^{d-j,j}_{C'} (A'_{C'} \otimes G^{d-j,j}) \subset A'_{C'} \otimes G^{d-j-1,j+1} \otimes \Omega^1_{C'}(\log (D'+S^{new}_{C'})).
\]
 Define $G:= \bigoplus^{m+1}_{j=0}G^{d-j,j}$. Then, we define the Higgs   bundle $(G,\tau_G)$ by letting the Higgs map $\tau_G$ be the restriction of $\mathrm{Id} \otimes \theta_{C'}$ to $G$. Note that $(G,\tau)$ is a sub-Higgs bundle of $A'^{-1}_{C'} \otimes E_{C'}$, as required. 
 
 We now invoke Griffiths's curvature computations for variations of Hodge structures. The necessary result is provided by the following general lemma (see  \cite[Proposition~2.4]{VZArIneq} for a detailed proof).
 
 \begin{lemma}[Griffiths]
 The determinant of a sub-Higgs bundle of a Hodge bundle is semi-negative.  (Here by ``Hodge bundle'' we mean the graded Higgs bundle associated to a polarized $\mathbb{C}$-VHS with respect to the Hodge filtration.)  
 \end{lemma} 

 We now continue the proof of Theorem \ref{thm:pointed_rigidity}.
Note that  the above   construction of $G$ implies that
\begin{eqnarray*}
\sum^m_{j=0} \mathrm{deg}_C \left( \mathcal{O}_C(B) \otimes T_C(-\log D)^{\otimes j} \right) &= &  \frac{1}{\deg \phi} \sum^m_{j=0} \mathrm{deg}_{C'} \left(\phi^\ast \left( \mathcal{O}_C(B) \otimes T_C(-\log D)^{\otimes j} \right)\right)  \\   
&\leq & \frac{1}{\deg \phi} \sum_{j=0}^m \deg_{C'} G^{d-j,j} 
=  \frac{1}{\deg \phi} \cdot \mathrm{deg}_{C'} G  .
\end{eqnarray*} 
 As $A'_{C'}$ is ample and $\mathrm{det}(A'_{C'} \otimes G)$ is semi-negative, we have that   $\mathrm{deg}_{C'}G <0$. Thus, we conclude that 
 \begin{eqnarray*}
\sum^m_{j=0} \mathrm{deg}_C \left( \mathcal{O}_C(B) \otimes T_C(-\log D)^{\otimes j} \right) &\leq &   
  \frac{1}{\deg \phi} \cdot \mathrm{deg}_{C'} G   <0.
\end{eqnarray*} The statement of the theorem now readily follows. Indeed, 
 since $$\deg_C T_C(-\log D) = 2-g(C) - \#D,$$ it follows that 
\[  (m+1) \# B + \frac{m(m+1)}{2} (2-2g(C)  - \# D) =\sum^m_{j=0} \mathrm{deg}_C \left( \mathcal{O}_C(B) \otimes T_C(-\log D)^{\otimes j} \right) <0. \]
In particular, the following inequality
\[
\# B < \frac{m}{2} \, (2g(C) - 2 + \# D)
\] holds.  As   $m\leq d-1$, we infer that
\[
\# B < \frac{m}{2} \,( 2g(C) -2 + \# D) \leq \frac{d-1}{2}\,( 2g(C) -2 + \# D).
\] This concludes the proof.
\end{proof}

\section{Proof of the Weak-Pointed Shafarevich conjecture}
In this section we prove Theorem \ref{thm:main} by combining our rigidity theorem for pointed curves (Theorem \ref{thm:pointed_rigidity}) with the recently established boundedness result   \cite[Theorem~D]{DLSZ} (generalizing earlier work of Kov\'acs-Lieblich \cite{KovacsLieblich}).

\begin{proof}[Proof of Theorem \ref{thm:main}] 
By a standard specialization argument (see, for instance, the arguments in   \cite[\S~9]{vBJK}), we  may and do assume that $k=\CC$.  Let $U$ be an integral variety over $\CC$ and $U\to \mathcal{M}_{h} \otimes_{\QQ} \CC$ be a quasi-finite morphism. Let $n:=\dim U$. We argue by induction on $n$. The statement is clear for $n=0$. Thus, we may and do assume that $n>0$. 

 Let $C$ be a smooth quasi-projective connected curve over $\CC$ with smooth projective model $\overline{C}$ and $D:=\overline{C}\setminus C$, let $$N\geq \frac{\deg h-1}{2} \left(2g(\overline{C}) - 2 + \# D\right)$$ be an integer, let $c_1,\ldots, c_N$ be pairwise distinct points in $C(\CC)$, and let $u_1,\ldots, u_N$ be points of $U(\CC)$. Let $B = \{c_1,\ldots, c_N\}$.
 
 To apply our main rigidity result (Theorem \ref{thm:main}), we choose data as in Section \ref{section:3}. Thus, 
  let $Y$ be a projective compactification of $U$ with $S:=Y\setminus U$. 
Let $\psi:\hat{Y}\to Y$ be a proper birational surjective morphism with $\hat{Y}$ a smooth projective variety      $\hat{S} = \psi^{-1}(S)$ is a simple normal crossings divisor, the inverse image $E$ of $\mathrm{Exc}(\psi)$ along $\psi:\hat{Y}\to Y$ is a simple normal crossings divisor and $\hat{S}\cup E$ is a simple normal crossings divisor.
Let  $X\to Y$ be a compactification of the composed map $V\to U\subset Y$,    and let $X\to Y$ be the pull-back   along $\psi:\hat{Y}\to Y$. We resolve $X$ to obtain  a  surjective proper morphism $g:\hat{X}\to \hat{Y}$ of smooth projective varieties  with $\hat{\Delta}:= g^{-1}(\hat{S})$ a normal crossings divisor  on $\hat{X}$, so that the following diagram commutes
\[
\xymatrix{V \ar[d]_{\pi} \ar[rr]^{\textrm{open immersion}} & &  {X} \ar[d]   & &\hat{X}  \ar[ll] \ar[d]^{g} \\ U \ar[rr]_{\textrm{open immersion}}  & &  {Y}  & & \hat{Y}.\ar[ll]^{\psi} }
\]
As in Section \ref{section:3}, we let $\mathrm{Supp}(T_g)$ be the support of the torsion subsheaf of  $  R^1g_*T_{\hat{X}/\hat{Y}}(- \log \hat{\Delta})$.

Let $U_1:=\mathrm{Exc}(\psi)\cup \mathrm{Supp}(T_g)$, and note that $U_1\subset U$ is a proper closed subscheme of $U$ with $\dim U_1 <\dim U$, where we endow $U_1$ with its reduced closed subscheme structure.  Moreover,    the composed map $U_1\subset U \to \mathcal{M}_{h,\CC}$ is quasi-finite.  In particular, every irreducible component of $U_1$   admits a quasi-finite morphism to $\mathcal{M}_{h,\CC}$. Since   $\dim  U_1 < \dim U = n$, it follows from   the induction hypothesis that   the set of non-constant morphisms  $f:C\to U$  with $f(c_1) = u_1, \ldots, f(c_N) = u_N$ and such that $f$  maps into $U_1$ is finite. (Indeed, if there were infinitely many such maps, as $U_1$ has only finitely many irreducible components,  there would be infinitely many such maps mapping to  some irreducible component of $ U_1$. We   then apply the induction hypothesis to this   component.)

 Thus, to prove the theorem, it suffices to  show that the set  of non-constant morphisms $\varphi:C\to U$  with $\varphi(c_1) = u_1, \ldots, \varphi(c_N) = u_N$ and such that $\varphi(C)\not\subset U_1$   is finite.  
To do so, we argue by contradiction and assume that there are infinitely many pairwise distinct non-constant morphisms $\varphi:C\to U $ with $\varphi(c_1) = u_1, \ldots, \varphi(c_N) =N$ and such that $\varphi(C) \not\subset U_1$.  For every such morphism $\varphi:C\to U$,  we let $\overline{\varphi}:\overline{C}\to Y$ be the unique morphism  extending $\varphi$.  Now, for every such $\varphi$,  we have that $\overline{\varphi}^{-1}(S) \subset D$. In particular,  as $D$ is finite,  there are infinitely many such morphisms $\varphi:C\to U$ such that   $\overline{\varphi}^{-1}(S) = D$.

Recall that $\underline{\Hom}(C,U)$ is a finite type scheme over $\mathbb{C}$ (see Theorem \ref{thm:boundedness}).  In particular,  by our assumption that there are infinitely many morphisms $\varphi:C\to U$ satisfying the above conditions,   there is a smooth affine curve $T$, a morphism $f:C\times T \to U$,  and a   lift $\hat{f}: \overline{C}\times T\to \hat{Y}$ of $f:\overline{C}\times T \to Y$ such   that the set of morphisms $f|_{C\times \{t\}}: C\to Y$ is infinite (i.e., $f$ is a truly varying family of morphisms) and
 
  \begin{itemize}
\item  for every $t\in T(\mathbb{C})$, the morphism $\overline{f_t}:= f|_{\overline{C}\times \{t\}}:\overline{C}\to Y$ is   non-constant    with $\overline{f_t}^{-1}(S) = D$,
\item  for every $t\in T(\mathbb{C})$, the image of $\overline{f_t}$ is  not contained in  $U_1$, and
\item    for every $s$ and $t$ in $T(\mathbb{C})$, we have that  $(\overline{f_s})|_B = (\overline{f_t})|_B$.   
\end{itemize} 

We claim that the image of $f:C\times T\to U$ is two-dimensional.  (In particular, the data above also forces $\dim U\geq 2$.) Indeed, let $D\subset U$ be its image. If $D$ is zero-dimensional, then every morphism $f_t$ would be constant contradicting the fact that the morphisms $f_t$ are non-constant.  Now, assume that $\dim D =1$.   As $U$ is hyperbolic,  it follows that $D$ is hyperbolic. In particular,  by a well-known generalization of De Franchis's theorem for hyperbolic curves (see \cite{Imayoshi}), the set of non-constant (dominant) morphisms $C\to D$ is finite. This contradicts the fact that the  morphisms $f_t = f|_{C\times \{t\}} :C\to D$ form  an infinite set of pairwise distinct non-constant morphisms.

We have thus shown that $f$  satisfies all of the assumptions of Section \ref{section:3}, so that we may apply    Theorem \ref{thm:pointed_rigidity} to conclude that $$\# B =  N < \frac{\deg h - 1}{2}( 2g(\overline{C})- 2 + \#D).$$ This contradicts our assumption that 
  $$\# B= N \geq \frac{\deg h - 1}{2}( 2g(\overline{C})- 2 + \#D)$$ and concludes the proof.
\end{proof}

\begin{remark}\label{remark:singular}
The proof of Theorem \ref{thm:main} (and also Theorem \ref{thm:pointed_rigidity}) simplifies a bit when $U$ is assumed to be smooth. For example, one does not need to argue by induction on the dimension of $U$ in the above proof, as one can simply choose an snc compactification of $U$ in this case. Also, with the notation of Section \ref{section:3}, the set $B_1 $ is empty, so that (as already mentioned in Remark \ref{remark:xi}) the construction of the map $\xi$ (in Section \ref{section:xi}) is simpler when $U$ is smooth.
\end{remark}

\begin{remark}[Arakelov-Parshin's theorem]\label{remark:ap} Let us show that Arakelov-Parshin's finiteness theorem follows from Theorem \ref{thm:main}. 
Let $g\geq 2$ be an integer and let $\mathcal{C}:=\mathcal{C}_g\otimes_\QQ k$ be the stack over $k$ of smooth proper connected curves of genus $g$. Note that $\mathcal{C} = \mathcal{M}_h$, where $h(t)= (2g-2) t + (1-g)$. Since $\deg h =1$, the integer $N:=0$ satisfies the inequality 
\[
 N\geq \frac{\deg h - 1}{2}\left( 2g(\overline{C})- 2 + \#D\right).
\]
Let $U$ be a  variety over $k$ and let $\varphi:U\to \mathcal{C}$ be a finite \'etale cover, e.g., $U$ is the moduli space  $\mathcal{C}_g^{[3]}$ of curves  of genus $g$ with level $3$ structure.  Let $C$ be a smooth connected curve over $k$, and recall that a morphism $C\to \mathcal{C}$ is non-isotrivial if the composed map $C\to \mathcal{C}\to \mathcal{C}^{coarse}$ is non-constant.    

We argue by contradiction. Thus, let $f_n:C\to \mathcal{C}$ be a sequence of  pairwise-distinct non-isotrivial morphisms.  For every $n$, let $D_n$ be a connected component of $C\times_{f_n, \mathcal{C}, \varphi} U$. Since each $D_n$ is a finite \'etale cover of $C$ of degree at most $\deg(\varphi)$, the set  of isomorphism classes of the curves $D_n$ is finite. Thus, replacing $f_n$ by a subsequence if necessary, we may and do assume that $D:=D_1 \cong D_2 \cong D_3\cong\ldots$. Now, by  Theorem \ref{thm:main} (with $N=0$),  the set of non-constant morphisms $D\to U$ is finite, so that (clearly) the collection of isomorphism classes of the morphisms $f_n$  is finite, contradicting our assumption above. 
\end{remark}
 
\section{The Persistence Conjecture for moduli spaces of polarized varieties}
Throughout this section, we let   $k$ be an algebraically closed field of characteristic zero.  We start with the following definition (see also \cite[Definition~4.1]{JAut}).
 
\begin{definition}[Mildly bounded varieties]\label{defn:mild_bounded}  
A variety $X$ over  $k$ is \emph{mildly bounded over $k$} if, for every smooth quasi-projective connected curve $C$ over $ {k}$, there are an integer $m$ and points $c_1,\ldots, c_m$ in $C( {k})$ such that, for every $x_1,\ldots, x_m$ in $X( {k})$, the set of morphisms $f:C\to X$ with $f(c_1) = x_1, \ldots, f(c_m) = x_m$ is finite.
\end{definition}

With this definition at hand, we get the following consequence of Theorem \ref{thm:main} for $\mathcal{M}$, where $\mathcal{M}$ denotes (as in the introduction) the moduli stack of   polarized varieties  with semi-ample canonical bundle over $\QQ$.
\begin{theorem}[The moduli space is mildly bounded]\label{thm:mod_is_mb}
Let $K$ be an algebraically closed field of characteristic zero,  let $U$ be a variety over $K$, and let $U\to \mathcal{M}\otimes_{\QQ} K$ be a quasi-finite morphism. Then $U$ is mildly bounded over $K$.
\end{theorem}
\begin{proof}
This is a straightforward consequence of Theorem \ref{thm:main} and Definition \ref{defn:mild_bounded}.    Indeed, let $C$ be a smooth quasi-projective connected curve over $K$ with compactly supported Euler characteristic $e(C)$. Since $U$ is integral, there is a polynomial $h\in \QQ[t]$ such that the quasi-finite morphism $U\to \mathcal{M}\otimes_{\QQ} K$ factors over a quasi-finite morphism $U\to \mathcal{M}_h\otimes_{\QQ} K$. Let $m$ be a positive integer such that 
$m\geq \frac{\deg h  - 1}{2} | e(C)|$,  let  $c_1,\ldots, c_m$ in $C(K)$ be pairwise distinct points, and let $u_1,\ldots,u_m\in U(K)$. Then, by Theorem \ref{thm:main}, the set of morphisms $f:C\to U$ with $f(c_1) =u_1, \ldots, f(c_m) = u_m$ is finite. It follows that   $U$ is mildly bounded over $K$.
\end{proof}
 
A variety over $\QQ$ with only finitely many $S$-integral points in a number field is expected to have only finitely many integral points valued in any finitely generated integral domain of characteristic zero (see the Persistence Conjecture \ref{conjper} for a precise statement). If the variety in question satisfies some additional ``boundedness'' property, then this persistence can in fact be proven. For example, 
the  Persistence Conjecture   holds  in the case of varieties  satisfying mild boundedness over all field extensions; see \cite[Theorem~1.6]{JAut} for a  proof of the following statement. 
 
\begin{theorem}\label{thm:mb}
Let $X$ be an arithmetically hyperbolic variety over $k$ such that, for every algebraically closed field extension $k\subset K$, the variety $X_K$ is mildly bounded over $K$. Then, for every algebraically closed field extension $L/k$, the variety $X_L$ is arithmetically hyperbolic over $L$.
\end{theorem}

Now, to prove Theorem \ref{thm2}, we   combine Theorem \ref{thm:mod_is_mb} and Theorem \ref{thm:mb}.

 \begin{proof}[Proof of Theorem \ref{thm2}] As in the statement of the theorem, let $k$ be an algebraically closed field of characteristic zero, and let $U$ be a variety over $k$ which admits a quasi-finite morphism $U\to \mathcal{M}\otimes_{\mathbb{Q}} {k}$. Let $k\subset L$ be an extension of algebraically closed fields of characteristic zero.  Note that, 
 by  Theorem \ref{thm:mod_is_mb},  for every algebraically closed field extension $K/k$, the  variety $U_K$ is mildly bounded over $K$. Therefore, if $U$ is arithmetically hyperbolic over $k$,  then it follows from Theorem \ref{thm:mb} that $U_L$ is arithmetically hyperbolic over $L$. This concludes the proof.
 \end{proof}

 \bibliography{refsci}{}
\bibliographystyle{alpha}

\end{document}